\documentclass[11pt, final]{amsart}
\usepackage{amsmath, amsxtra, amssymb, color}
\usepackage{verbatim}
\usepackage{graphicx}
\usepackage{multirow, tikz}
\usepackage[margin=1.34in]{geometry}
\usepackage{amssymb, mathdots}

\usepackage{graphicx}

\usepackage[cmtip,all]{xy}

\usepackage{xypic}

\definecolor{linkred}{rgb}{0.7,0.2,0.2}
\usepackage[colorlinks, 
linkcolor=blue, 
citecolor= linkred,
pagebackref=true]{hyperref}

\theoremstyle{plain}

\numberwithin{equation}{section}

\newtheorem{theorem}{Theorem}[section]
\newtheorem{prop}[theorem]{Proposition}
\newtheorem{corollary}[theorem]{Corollary}

\newtheorem{lemma}[theorem]{Lemma}
\newtheorem*{mainresult}{Main Theorem}

\theoremstyle{definition}

\newtheorem{definition}[theorem]{Definition}

\newtheorem{remark}[theorem]{Remark}

\def\ra{\rightarrow}

\newcommand{\gitq}{/\hspace{-0.25pc}/}
\renewcommand{\ss}{\mathrm{ss}}

\def\dra{\dashrightarrow}
\def\co{\colon\thinspace} 

\DeclareMathOperator{\spec}{Spec}
\DeclareMathOperator{\proj}{Proj}
\DeclareMathOperator{\SHom}{\mathit{\mathcal{H}om}}

\DeclareMathOperator{\Hilb}{Hilb}
\DeclareMathOperator{\Tl}{\mathcal{T}}

\DeclareMathOperator{\Def}{Def}

\def\Hn1{\mathcal{H}_{n,1}}

\def\C{\mathcal{C}}
\def\D{\mathcal{D}}

\def\O{\mathcal{O}}

\def\M{\overline{M}}
\newcommand\Mg[1]{\overline{\mathcal{M}}_{#1}}

\def\N{\mathcal{N}}

\def\X{\mathcal{X}}

\def\ZZ{\mathbb{Z}}

\def\PP{\mathbb{P}}
\def\ZZ{\mathbb{Z}}

\def\FF{\mathbb{F}}

\def\CC{\mathbb{C}}

\def\HH{\mathrm{H}}

\DeclareMathOperator\SL{SL}
\DeclareMathOperator\PGL{PGL}
\DeclareMathOperator\Aut{Aut}
\DeclareMathOperator\Bl{Bl\,}

\def\nb{\nobreakdash}

\begin{document}
\title[The final log canonical model of $\M_4$]{The final log canonical model of 
the moduli space of stable curves of genus four}
\author{Maksym Fedorchuk}
\address{Department of Mathematics, Columbia University, 2990 Broadway, New York, NY 10027}
\curraddr{}
\email{mfedorch@math.columbia.edu}
\date{}

\begin{abstract}
We describe the GIT quotient of the linear system of $(3,3)$ curves on 
$\PP^1\times\PP^1$ as the final non-trivial log canonical model of $\M_4$, 
isomorphic to $\M_4(\alpha)$ for $8/17<\alpha\leq 29/60$. We describe singular curves parameterized by 
$\M_4(29/60)$, and show that the rational map $\M_4\dashrightarrow \M_4(29/60)$ 
contracts the Petri divisor, in addition to the boundary divisors $\Delta_1$ and $\Delta_2$. 
This answers a question of Farkas.
\end{abstract}

\maketitle

\section{Introduction}
The goal of this note is twofold. One is to show
that the Petri divisor on $\M_4$ is contracted by 
a rational contraction, thus answering a question of Farkas.  
Second is to describe the final non-trivial log canonical model 
appearing in the Hassett-Keel log minimal model program for $\M_4$, thus 
confirming various predictions obtained in \cite{afs} for when 
singular curves replace curves with special linear systems.

We now describe each of these goals in more detail: It is well-known 
that the hyperelliptic divisor in $\M_{3}$ is contracted
by the rational map to the final non-trivial log canonical model of $\M_{3}$
given by the GIT quotient of plane quartics;
see \cite{hyeon-lee_genus3}. Farkas has observed 
\cite[p.281]{farkas-gieseker} that on $\M_{4}$ 
there is no rational contraction, well-defined away from 
the hyperelliptic locus, that contracts the Petri divisor $P\subset \M_4$. 
Subsequently, Farkas asked \cite{farkas-personal} 
whether there are rational contractions, 
necessarily with a larger indeterminacy locus, 
that do contract $P$. Here, we answer this question in affirmative:
The rational map to the GIT quotient of $(3,3)$ curves on 
$\PP^1\times\PP^1$ contracts the Petri divisor to a point. This map  
is undefined both along the hyperelliptic locus {\em and} the locus of irreducible nodal curves with 
a hyperelliptic normalization. 

Our second goal is to describe the final non-trivial step in the Hassett-Keel log MMP
program for $\Mg{4}$ and to verify that it satisfies the modularity principle of \cite{afs}. 
The aim of the Hassett-Keel program for $\Mg{g}$ is to find an open substack $\Mg{g}(\alpha)$ in 
the stack of all complete genus $g$ curves such that $\Mg{g}(\alpha)$ has a good moduli space isomorphic
to the log canonical model 
\begin{align*}
\M_g(\alpha):=\proj \bigoplus_{m \ge 0} \HH^0(\Mg{g}, \lfloor m(K_{\Mg{g}} + \alpha \delta)\rfloor).
\end{align*}

As of this writing, the Hassett-Keel program for $\Mg{4}$ has been 
carried out for $\alpha\geq 2/3$ using GIT of 
the Hilbert and Chow schemes of bicanonically embedded curves 
and the threshold values
at which $\M_4(\alpha)$ changes are  $\alpha=9/11, 7/10, 2/3$
\cite{hassett-hyeon_flip, hassett-hyeon_contraction, hyeon-lee_genus4}. 
In particular, Hyeon and Lee \cite{hyeon-lee_genus4} construct a small contraction $\M_4(7/10-\epsilon) \ra \M_4(2/3)$
of the locus of Weierstrass genus $2$ tails (i.e. curves $C_1\cup_p C_2$ where $C_1$ and $C_2$ are genus two curves meeting in a node 
$p$ such that $p$ is a Weierstrass point of $C_1$ or $C_2$) using Kawamata basepoint freeness theorem. 
An alternative approach to the Hassett-Keel program for $\Mg{g}$ and to a functorial construction of the log canonical model $\M_g(2/3)$  
is pursued by Alper, Smyth, and van der Wyck in \cite{asw}. They
define a moduli stack $\Mg{g}(A_4)$ of genus $g$ curves with at worst $A_4$ singularities and no Weierstrass genus $2$ tails and show
that it is weakly proper without using GIT.
Once the existence and the projectivity of a good moduli space of $\Mg{g}(A_4)$ is established, 
$\Mg{g}(A_4)$ will give a modular interpretation of the log canonical model
$\M_g(2/3)$ for every $g\geq 4$.

We note that there are other threshold values at which $\M_{4}(\alpha)$ changes for $0<\alpha< 2/3$; these 
can be easily obtained from \cite{afs}. One of them is $\alpha=5/9$ and the corresponding log canonical model has been completely 
described by Casalaina-Martin, Jensen, and Laza \cite{laza-et-al} as a GIT quotient of the Chow variety of canonically embedded genus $4$ curves. 
In a forthcoming work \cite{laza-et-al-2}, the same authors describe all log canonical models $\M_4(\alpha)$
that arise for $29/60<\alpha<5/9$ by using VGIT 
on the parameter space of $(2,3)$ complete intersections in $\PP^3$ and show
that the only threshold value in the interval $(29/60, 5/9)$ is $\alpha=1/2$, which agrees with 
predictions of \cite{afs}. 

Here, we describe the final step in the Hassett-Keel program for 
$\Mg{4}$ which 
is given by a natural GIT quotient of the linear system of $(3,3)$ curves on $\PP^1\times \PP^1$:
A starting point is the classical observation 
that a canonical embedding of a non-hyperelliptic smooth curve 
of genus $4$ lies on a unique 
quadric in $\PP^3$. If the quadric is smooth, 
the curve is called {\em Petri-general}, 
and is realized as an element of the linear system
$$V:=\vert\O_{\PP^1\times\PP^1}(3,3)\vert\simeq \PP^{15}$$
of $(3,3)$-curves on $\PP^1\times\PP^1$. 
Moreover, the uniqueness of a pair of $g^1_3$'s implies
that two smooth curves of class $(3,3)$ are
abstractly isomorphic if and only if they belong to the same 
$\Aut(\PP^1\times\PP^1)$-orbit. 
Conversely, a smooth genus $4$ curve is called {\em Petri-special} if its canonical image lies
on a singular quadric. Petri-special curves form a divisor whose
closure in $\Mg{4}$ is called the {\em Petri divisor};  we denote it by $P$. 

This said, we consider a linearly reductive group $G=\SL(2)\times\SL(2)\rtimes \ZZ_2$ that, while being 
a finite cover of $\Aut(\PP^1\times\PP^1)$, has the advantage of linearizing $\O_V(1)$. 
Then the GIT quotient $V^{\ss} \gitq G$ will be a birational model of $\M_4$ as soon as 
the general curve in $V$ is GIT stable, which is easy to verify. 
Our main result is:
\begin{mainresult}
The GIT quotient 
$M:=\vert\O_{\PP^1\times\PP^1}(3,3)\vert^{\ss} \gitq (\SL(2)\times\SL(2)\rtimes \ZZ_2)$
is isomorphic to $\M_{4}(\alpha)$ for $8/17<\alpha\leq 29/60$. 
The resulting 
birational contraction $f\co \M_4 \dra \M_{4}(29/60)$ contracts the Petri divisor $P$ to the point 
parameterizing triple conics, 
and contracts the boundary divisor $\Delta_2$ to the point 
parameterizing curves with two $A_5$ singularities. The 
hyperelliptic locus $\overline{H}_4$ is flipped by $f$ to the locus
{\em $$A:=\overline{\{\text{curves with an $A_8$ singularity}\}},$$}
i.e. the total transform of the generic point of $\overline{H}_4$ is $A$.
\end{mainresult}
Main Theorem is proved in Section \ref{S:proof-main-theorem}. We give a roadmap to its proof: 
That $f$ is a contraction is proved in Proposition \ref{P:contraction}.
That $f$ contracts $P$ and $\Delta_2$ to a point 
is proved in Theorem \ref{T:determinacy}. 
That $\overline{H}_4$ is flipped to $A$ is established in Theorem \ref{T:flip}. Finally, 
the identification of $M$ with log canonical models of $\M_4$ is made in Corollary \ref{logMMP}.

We also obtain a strengthening of the genus $4$ case of 
\cite[Theorem 5.1]{farkas-gieseker},
whose terminology we keep:
\begin{theorem}\label{T:moving-slope}
The moving slope of $\M_4$ is $60/7$.
\end{theorem}
Theorem \ref{T:moving-slope} is proved in Corollary \ref{C:moving-cone}.

\begin{remark} A birational contraction of the Petri divisor inside $\M_{4,1}$ is
constructed in \cite{jensen-genus4} using GIT on the universal curve over $V$.
However, Jensen does not give a modular interpretation to this contraction.
\end{remark}

\subsection{Preliminaries}
We recall some 
definitions and results that will be used throughout this work.
We work over $\CC$.
\subsubsection{Varieties of stable limits}
Let $C$ be an l.c.i. integral curve of arithmetic genus $g$ 
and $p\in C$ be a singular point. Recall that the {\em variety of 
stable limits of $C$} is the closed subvariety 
$\Tl_{C}\subset \M_g$ consisting 
of stable limits of  all possible smoothings of $C$. Namely,
if $\Def(C) \stackrel{p}{\longleftarrow} Z 
\stackrel{q}{\longrightarrow} \M_g$ is the graph of the rational moduli 
map $\Def(C)\dra \M_g$, 
then we set $\Tl_{C}:=q(p^{-1}(0))$. 

Suppose that $p\in C$ is the only singularity. Let $b$ be the number of
branches of $p\in C$ and 
$\delta(p)=\dim_{\CC} \O_{\widetilde{C}}/\O_{C}$ be the $\delta$\nb-invariant.
Then curves in $\Tl_{C}$ are of the form $\widetilde{C} \cup T$, 
where $(\widetilde{C}, q_1, \dots, q_b)$ is the pointed normalization of $C$ and 
$(T,p_1,\dots, p_b)$ is a $b$-pointed curve of arithmetic genus 
$\gamma=\delta(p)-b+1$. 
The pointed stable curve
$(T, p_1,\dots, p_b)$ is called the {\em tail of a stable limit}. 
Tails of stable limits are independent of $\widetilde{C}$
and depend only on $\hat{\O}_{C,p}$. 
It follows that we can define the variety of tails of stable limits of 
$\hat{\O}_{C,p}$ as a closed subvariety 
$\Tl_{\hat{\O}_{C,p}}\subset \M_{\gamma,b}$ 
(see \cite[Proposition 3.2]{Hassett-stable}). 

We recall the following results concerning 
the varieties of tails of stable limits of $A$ and $D$ singularities
(see \cite[Sections 6.2, 6.3]{Hassett-stable} and \cite{fedorchuk-AD}).
\begin{prop}[Varieties of stable limits of AD singularities]
\label{P:stable-limits}
\begin{enumerate}
\item[]
\item[$\mathrm{(A_{\text{odd}})}$]\label{Aodd} The variety of 
tails of stable limits of the $A_{2k+1}$ singularity $y^2=x^{2k+2}$ 
is the locus of $(C,p_1,p_2) \in \M_{k, 2}$ 
such that a semistable model $C'$ of $C$ admits
an admissible hyperelliptic cover $\varphi\co C'\ra R$, where $R$ is a rational nodal curve, 
and $\varphi(p_1)=\varphi(p_2)$.
\item[$\mathrm{(A_{\text{even}})}$]\label{Aeven}
The variety of 
tails of stable limits of the $A_{2k}$ singularity $y^2=x^{2k+1}$ 
is the locus of $(C,p) \in \M_{k, 1}$ such that a semistable model $C'$ of $C$ admits
an admissible hyperelliptic cover $\varphi\co C'\ra R$, where $R$ is a rational nodal curve, 
and $p$ is a ramification point of $\varphi$.
\item[$\mathrm{(D_{\text{odd}})}$] The variety of 
tails of stable limits of the $D_{2k+1}$ singularity $x(y^2-x^{2k-1})=0$ 
is the locus $(C,p_1,p_2) \in \M_{k, 2}$ such that 
$p_1\neq p_2$
and the stabilization of $(C,p_1)$ is as in 
{\em ($\mathrm{A}_{\text{even}}$)}.
\item[$\mathrm{(D_{\text{even}})}$]
The variety of 
tails of stable limits of the $D_{2k}$ singularity $x(y^2-x^{2k-2})=0$ 
is the locus $(C,p_1,p_2, p_3) \in \M_{k-1, 3}$ such that 
$p_3\notin \{p_1,p_2\}$
and the stabilization of $(C,p_1,p_2)$ is as in 
{\em ($\mathrm{A}_{\text{odd}}$)}.

\end{enumerate}
\end{prop} 
\begin{proof}
This is the content of 
\cite[Main Theorem 1(2) and Main Theorem 2(2)]{fedorchuk-AD}.
\end{proof}

\subsubsection{Deformations of curves on $\PP^1\times\PP^1$}
Suppose that $p_1,\dots, p_n$ are singular points of a reduced curve 
$C$. Then 
\begin{equation}\label{E:versality}
\Def(C)\ra \prod_{i=1}^{n} \Def(\hat{\O}_{C,p_i})
\end{equation}
is smooth because $\HH^2(C, \SHom(\Omega_C, \O_C))=(0)$. If $C$ is a $(3,3)$ curve
on $\PP^1\times \PP^1$, 
then in fact all deformations of $\hat{\O}_{C,p_i}$ are realized by embedded deformations of $C$:
\begin{prop}\label{P:versality} Let $C$ be a reduced curve in 
class $(3,3)$ on $\PP^1\times\PP^1$ and $p_1,\dots, p_n$ are singular points of $C$.
Then the natural map 
$\Hilb(\PP^1\times\PP^1) \ra \prod_{i=1}^{n} \Def(\hat{\O}_{C,p_i})$
is smooth.
\end{prop}
\begin{proof} This is standard. Set $X=\PP^1\times\PP^1$. Since
$\Hilb(X)$ and $\prod_{i=1}^{n} \Def(\hat{\O}_{C,p_i})$ can be taken to be of finite type over $\CC$, it 
suffices to establish formal smoothness. 
Since $C\hookrightarrow X$ is 
locally unobstructed, it suffices to check that a collection 
of local first-order deformations can always be glued to a global embedded
deformation. A sufficient 
condition for this is the vanishing of $\HH^1(C, \N_{C/X})$ and $\HH^1(C, T_{X}\otimes \O_{C})$. 
We now compute:
\begin{align*}
\HH^1(C,\N_{C/X})&=\HH^0(C,\omega_{C}\otimes \N_{C/X}^\vee)^\vee 
 =\HH^0(C, \O(-2,-2)\vert_{C})^\vee=(0),\\
 \HH^1(C, T_{X}\otimes \O_{C})&=\HH^1(C, \O_C(2,0)\oplus \O_C(0,2))
 =\HH^0(C, \O_C(-1,1)\oplus \O_C(1,-1))^\vee=(0).
\end{align*}
\end{proof}

\section{GIT of $(3,3)$ curves on $\PP^1\times \PP^1$}\label{S:GIT}

In this section, we classify semistable points of
$V:=\vert \O_{\PP^1\times\PP^1}(3,3) \vert$ under the 
action of $G:=(\SL(2)\times\SL(2))\rtimes \ZZ_2$ 
by applying the Hilbert-Mumford numerical criterion
\cite[Chapter 2.1]{GIT}. In Section \ref{S:numerical-criterion}
we describe equations of (semi)stable and nonsemistable points. 
The geometric consequences 
of these results are then collected in Section \ref{S:analysis}.
\subsection{Numerical criterion}\label{S:numerical-criterion}
Choose projective 
coordinates $X,Y, Z,W$ on $\PP^1\times\PP^1$.
Consider then one-parameter subgroup  
$\rho_{u,v}\co \spec \CC[t,t^{-1}]\ra G$ that acts via 
\[
t\cdot (X,Y,Z,W)=(t^uX,t^{-u}Y,t^vZ,t^{-v}W),
\] where we assume that 
$u\geq v\geq 0$. 

With respect to $\rho_{u,v}$, the monomial $X^iY^{3-i}Z^jW^{3-j}$ 
has weight $(2i-3)u+(2j-3)v$. The monomials of positive weight are those with:
\begin{enumerate}
\item $i=3, j\geq 1$. 
\item $i=3, j=0$ if $u>v$. 
\item $i=2, j\geq 2$. 
\item $i=2, j=1$  if $u>v$. 
\item $i=2, j=0$ if $u>3v$. 
\item $i=1, j=3$ if $u<3v$. 
\end{enumerate}
It follows that the general {\em nonsemistable}
point is nonsemistable either 
with respect to the one-parameter subgroup $\rho_{2,1}$ or 
with respect to $\rho_{4,1}$. We record that in the affine coordinates
$x=X/Y$ and $z=Z/W$
the general nonsemistable point with respect to $\rho_{2,1}$ 
has equation 
\begin{equation}\label{E:nonstable-21}
x(cz^3+a_0xz+a_1xz^2+a_2xz^3+b_1x^2+b_2x^2z+b_3x^2z^2+b_4x^2z^3)=0.
\end{equation}
Similarly, the general nonsemistable point with respect to $\rho_{4,1}$ is
\begin{equation}\label{E:nonstable-41}
x^2(1+a_0z+a_1z^2+a_2z^3+b_1x+b_2xz+b_3xz^2+b_4xz^3)=0.
\end{equation}

We proceed to describe strictly semistable points.
It is clear from the above list of monomials that the only
one-parameter subgroups $\rho_{u,v}$ with respect to which 
there are monomials of degree $0$ are $\rho_{1,1}$ and 
$\rho_{3,1}$. 
The degree $0$ monomials with respect to $\rho_{3,1}$ are
\begin{enumerate}
\item $X^2YW^3$, which become positive if $u>3v$.
\item $XY^2Z^3$, which becomes negative if $u>3v$.
\end{enumerate} 
Since there are only two monomials of weight $0$ when $u=3v$, any curve which is strictly semistable
with respect to $\rho_{3,1}$ and has a closed orbit is unique up to automorphisms of $\PP^1\times\PP^1$
and is defined by the equation
\begin{equation}\label{E:semistable31}
X^2YW^3+XY^2Z^3=XY(XW^3+YZ^3)=0.
\end{equation}
The degree $0$ monomials with 
respect to $\rho_{1,1}$ are
\begin{enumerate}
\item $X^3W^3$, 
which becomes positive for $u>v$ and negative for $u<v$.
\item $X^2YZW^2$, 
which becomes positive for $u>v$ and negative for $u<v$.
\item $XY^2Z^2W$, 
which becomes negative for $u>v$ and positive for $u<v$.
\item $Y^3Z^3$, 
which becomes negative for $u>v$ and positive for $u<v$.
\end{enumerate}
Thus, strictly semistable points with respect to $\rho_{1,1}$ 
have form
\begin{equation*}
ax^3+bx^2z+cxz^2+dz^3+ex^2z^2+fxz^3+gx^2z^3+hx^3z^3=0
\end{equation*}
Following such curves to the flat limit under $\rho_{1,1}$, we 
see that every strictly semistable point isotrivially specializes to a curve $ax^3+bx^2z+cxz^2+dz^3=0$, 
or, in projective coordinates, to a curve
\begin{equation}\label{E:semistable-11-closed}
L_1L_2L_3=aX^3W^3+bX^2W^2ZY+cXWZ^2Y^2+dZ^3Y^3=0,
\end{equation}
where $(a,b)$ are not simultaneously zero and $(c,d)$ are not simultaneously zero. 
(Here, $L_i$ are homogeneous forms of bidegree $(1,1)$.)
\begin{prop}\label{P:D-curves-1}
The orbit closure of every semistable curve with a multiplicity $3$ singularity 
contains its ``tangent cone,'' which is 
a curve described by Equation \eqref{E:semistable-11-closed}.
In particular, the orbit closure of every semistable curve with a 
$D_8$ singularity 
contains a double conic, i.e.
a curve defined by $(ax+bz)^2(cx+dz)=0$.
\end{prop}
\begin{proof} Let $C$ be a semistable curve with a multiplicity $3$ singularity.
Choose coordinate $X,Y,Z,W$ so that the equation of $C$ in the affine coordinates $x:=X/Y$ and
$z:=Z/W$ is $f_3+f_{\geq 4}=0$, where $f_3$ is a homogeneous polynomial of degree $3$ in $x$ and $z$,
and $f_{\geq 4}\in (x,z)^4$. 
Since $C$ is semistable, we have that $x^2, z^2\nmid f_3$;
otherwise $C$ would be defined by Equation \eqref{E:nonstable-21}.
Then the flat limit of $C$ under $\rho_{1,1}$ 
is $f_3(x,z)=0$ or, in projective coordinates, 
$$
a X^3W^3+bX^2W^2ZY+cXWZ^2Y^2+dZ^3Y^3=0,
$$
where $(a,b)$ are not simultaneously zero and $(c,d)$ are not simultaneously zero.

If the triple point is a $D_8$ singularity, the tangent cone
is a union of a double line and a transverse line, i.e. $f_3(x,z)=(ax+bz)^2(cx+dz)$. 
\end{proof}

\subsection{Stability analysis}\label{S:analysis}
We summarize the calculations of the previous section and reinterpret them in a geometric language.
First, we describe nonsemistable curves. 

\begin{prop}[Nonsemistable curves]\label{P:unstable}
A curve $C$ is nonsemistable if and only if one of the following holds: 
\begin{enumerate}
\item $C$ contains a double ruling.
\item $C$ contains a ruling and the residual curve $C'$ intersects this ruling in a 
unique point that is also a singular point of $C'$. 
\end{enumerate}
\end{prop}
\begin{proof}
Nonsemistable curves are precisely those curves that are defined by Equations  \eqref{E:nonstable-21} 
and \eqref{E:nonstable-41} for some choice of coordinates.
Equation \eqref{E:nonstable-21} 
defines a reducible curve $C=C_1\cup C_2$, where 
$C_1$ is a ruling of $\PP^1\times \PP^1$ 
that intersects the residual $(2,3)$ curve $C_2$ with multiplicity $3$ at the singular point $x=z=0$
of $C_2$. Finally, Equation \eqref{E:nonstable-41} defines a curve with a double ruling.  
\end{proof}
\begin{corollary}\label{C:nonreduced}
The only non-reduced semistable curves are:
\begin{enumerate}
\item Triple conics; these are strictly semistable and have closed orbits.
\item A union of a smooth double conic and a 
conic which is nonsingular along the double conic; these are strictly semistable and have closed orbits.
\end{enumerate}
\end{corollary}
\begin{proof}
The proof is immediate: By Proposition \ref{P:unstable} 
and degree considerations, 
any non-reduced structure has to be supported
along a smooth conic. If the generic multiplicity is $3$, then 
the curve is a triple conic. If it is $2$, then the residual curve cannot 
be a union of two rulings meeting along the double conic by {\em loc. cit.}
\end{proof}
\begin{remark}[Double conics]\label{R:cross-ratio}
For brevity we call the curve described in Part (2) of Corollary \ref{C:nonreduced} a {\em double conic}.
Given a double conic $C=2C_1+C_2$, we consider the horizontal rulings $L_1$ and $L_2$ passing
through two triple points $C_1\cap C_2$. Then the four points of intersection of 
the general vertical ruling with $C_1,C_2,L_1,L_2$ define a cross-ratio 
which is a $\Aut(\PP^1\times\PP^1)$-invariant of $C$. We call it 
the {\em cross-ratio of the double conic $C$}.
\end{remark}

Any other curve that does not fit the description of Proposition \ref{P:unstable} 
is semistable. Recall that there is a unique closed orbit of
semistable curves which is strictly semistable with respect to $\rho_{3,1}$.
This curve is defined by Equation \eqref{E:semistable31}.
We call such a curve the {\em maximally degenerate
$A_5$-curve.} It consists of two lines in the same ruling $(1,0)$ and a smooth $(1,3)$ curve meeting each 
line at an $A_5$ singularity (i.e. tangent with multiplicity $3$); see Figure \ref{F:A5-A5}.

\begin{figure}[hbt]
\begin{center}
\includegraphics[scale=1]{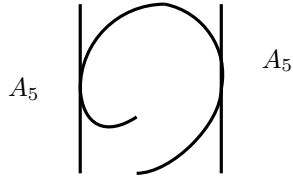}
\end{center}
\vspace{-1pc}
\caption{A maximally degenerate $A_5$-curve}\label{F:A5-A5}
\end{figure}

Recall from above that every semistable curve which is strictly
semistable with respect to $\rho_{1,1}$ and has a closed orbit is given by Equation \eqref{E:semistable-11-closed}.
Every such curve is a union of $3$ conics in the class $(1,1)$, all meeting at points $0\times0$ and 
$\infty\times \infty$. In addition, by scaling $X$ and $Z$, we can assume that either
$a=d=1$, or $a=d=0$ and $b=c=1$. 
This leaves us with a $2$-dimensional family of strictly semistable points
$$x^3+bx^2z+cxz^2+z^3=0,$$
or $x^2z+xz^2=xz(x+z)=0.$
We call such curves {\em D-curves}, because the generic D-curve has two ordinary triple point 
($D_4: x^3=y^3$) singularities. 

We can now restate Proposition \ref{P:D-curves-1} in a geometric language:
\begin{prop}\label{P:D-curves-2}
The orbit closure of every semistable curve with a multiplicity $3$ singularity contains a D-curve described by Equation \eqref{E:semistable-11-closed}, 
i.e. either a union of three conics at two $D_4$ singularities, or 
a double conic, or a triple conic.
\end{prop}
\begin{remark}
We also note that every non-reduced semistable curve is a D-curve: {\em triple conics} 
arise from Equation \ref{E:semistable-11-closed} by taking $L_1=L_2=L_3$ and {\em double conics} 
arise by taking $L_1=L_2$.
\end{remark}
\subsection{Geometry of semistable curves}
We refine the GIT analysis to obtain a list of geometrically meaningful strata inside the semistable locus.
We begin with strictly semistable points in the highest stratum:
\subsubsection{D-curves\em:}
\label{two-D4} A D-curve is a curve defined by Equation 
\eqref{E:semistable-11-closed}. These are precisely strictly semistable
curves (with closed orbits) which consist of three conics in class $(1,1)$ 
passing through two points of $\PP^1\times \PP^1$ not 
on the same ruling, i.e. three conics meeting in two $D_4$ (ordinary 
triple points) singularities (see Figure \ref{F:D-curves}).
By Proposition \ref{P:stable-limits} ($\mathrm{D}_{\text{even}}$) and Proposition \ref{P:versality}, 
the variety of stable limits of the general D-curve is the locus of
{\em elliptic triboroughs} in $\M_4$, 
i.e. nodal unions of two elliptic components
along three points.

\begin{figure}[hbt]
\begin{center}
\includegraphics[scale=0.75]{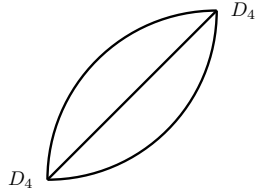}
\end{center}
\vspace{-0.5pc}
\caption{A D-curve}\label{F:D-curves}
\end{figure}

\subsubsection{Curves with separating $A_5$ singularities\em:} 
Such curves are necessarily of the form
$C=C_1\cup C_2$ where $C_1$ is a ruling in class $(1,0)$  
and $C_2$ is a curve in class $(2,3)$ that intersects $C_1$ with multiplicity $3$ at a smooth point of $C_2$. 
It is easy to see that all such curves contain the maximally degenerate
$A_5$-curve in their orbit closure. 
From Proposition \ref{P:stable-limits} $\mathrm{(A_{\text{odd}})}$ and Proposition \ref{P:versality},
we conclude that
the variety of stable limits of the maximally degenerate
$A_5$-curve is all of $\Delta_2\subset \M_4$.

\subsubsection{Double conics\em:}\label{double-conics} 
These are defined by the equation $L_1^2L_2=0$, where $L_1$ is an irreducible form of bidegree $(1,1)$ and $L_2$ is a form of bidegree $(1,1)$ that meets $L_1$ in two distinct points. 
Double conics form a closed locus inside the locus of D-curves. 
\subsubsection{Curves with $D_8$ singularities\em:} \label{S:D-8}
Consider a curve $C=C_1\cup C_2$, where $C_2$ is a nodal curve in class $(2,2)$ and $C_1$ is a smooth curve in class $(1,1)$ that intersects one of the branches of the node of $C_1$ with multiplicity $3$.
These curves do not have closed orbits: they isotrivially specialize to 
double conics by Proposition \ref{P:D-curves-1}. The reason we single out this class of curves is that 
it follows immediately from Proposition \ref{P:stable-limits} $(\mathrm{D_{\text{even}}})$ that the variety of 
stable limits of a $D_8$-curve is the closure of the locus of irreducible nodal curves with a hyperelliptic normalization. 
Denote this locus by $\Delta_0^{hyp}$. Then $\Delta_0^{hyp}$ is divisorial inside the Petri divisor $P$: it is the locus of canonical genus $4$ 
curves lying on a singular quadric and passing through its vertex. 
Note that $\Delta_0\cap P$ has two irreducible components, with $\Delta_0^{hyp}$ being one of them.

Since $D_8$-curves specialize isotrivially to double conics, $\Delta_0^{hyp}$ also lies in the variety of stable
limits of double conics. Observing that double conics have moduli (see Remark \ref{R:cross-ratio}), we conclude
that the rational map from $\M_4$ to the GIT quotient $V^\ss\gitq G$ is undefined along $\Delta_0^{hyp}\subset P$.

\subsubsection{Triple conics\em:}\label{triple-conics} These form
a single (closed) orbit of curves 
defined by the equation $L^3=0$, where $L$ is an irreducible form of bidegree $(1,1)$. 
The corresponding point in $V^{\ss}\gitq G$ lies in
the closure of double conics. 
Among semistable curves whose orbit closure contains the orbit of the triple conic are curves
with a $J_{10}$ singularity $y^3=x^6$ defined 
by the equation $L_1L_2L_3=0$, where $L_i$ are forms of bidegree $(1,1)$ such 
that the corresponding conics all meet in a single point:
$$
(x+z+c_0xz)(x+z+c_1xz)(x+z+c_2xz)=0.
$$
Evidently, all such curves isotrivially specialize to the triple conic.
\subsubsection{Curves with $E_6, E_7, E_8$ singularities\em:}
None of these have closed orbits:
These arise as deformations of curves with $J_{10}$ 
singularities (see \cite[Section I.1]{arnold-inventiones} or \cite{J10})
and in fact isotrivially specialize to the triple conic by Proposition 
\ref{P:D-curves-1}. (Note, however, that the variety of stable limits of $E_6$ is 
the locus of $[C_1\cup C_2] \in \Delta_1$ such that $C_1\cap C_2$ is a hyperflex
of the genus $3$ curve $C_2$; see \cite[Theorem 6.2]{Hassett-stable}.) 

\medskip

We proceed to describe geometrically meaningful strata inside the stable locus:
\subsubsection{Curves with $A_8$ and $A_9$ singularities\em:}
\label{S:A-curves}
Consider an $A_8$-curve $C$ (see \cite{FedJen} for a general background on canonical $A$-curves)
defined  
parametrically by
\[
t\mapsto [1-3t+3t^3, (1-3t+3t^3)(t^2+t^3), t^2(1-2t), t^4(1-t-2t^2)].
\]
This curve is a complete
intersection of the smooth quadric $z_0z_3=z_1z_2$ and
a cubic  in $\PP^3$. The only singularity of $C$ is of type $A_8$ and $C$ has a rational normalization.
Locally around $[C]$, the locus of curves with $A_8$ singularities is
the fiber of a smooth map $\Hilb(\PP^1\times\PP^1)\ra \Def(A_8)$ 
(see Proposition \ref{P:versality}). 
Denote by $A$ the closure of the locus of $A_8$-curves 
in $V^\ss\gitq G$.
Counting dimensions, we conclude 
that $\dim(A)=1$. 
Since $J_{10}$ singularity deforms to $A_8$ by \cite{J10}, we see
that $A$ passes through the triple conic.

By Proposition \ref{P:stable-limits} $\mathrm{(A_{\text{even}})}$, 
the variety of stable limits of an $A_8$-curve 
is the hyperelliptic locus $\overline{H}_4\subset\M_4$.
We will see in Theorem \ref{T:flip} that 
$f\co \M_4 \dra V^\ss\gitq G$ flips $\overline{H}_4$ to $A$.


There is another distinguished point in $A$, which corresponds to 
a union of $(2,1)$ and $(1,2)$
curve at an $A_9$ singularity. Up to projectivities, 
there is a unique $A_9$-curve. It is defined parametrically
by 
\[
\left(\begin{matrix} 1 & s & s^2 &s^3 \\
1-3t+3t^2+t^3 & t(1-2t+t^2) & t^2-t^3 & t^3 \end{matrix}\right).
\]
The variety of stable limits of the $A_9$-curve is $\overline{H}_4\subset\M_4$ 
by Proposition \ref{P:stable-limits} $\mathrm{(A_{\text{odd}})}$.
\subsubsection{Curves with $A_6$ or $A_7$ singularities\em:}
\label{S:A67-curves}
Curves with non-separating $A_7$ singularities ($y^2=x^8$) 
replace curves in $\Delta_0\cap \overline{H}_4$, 
i.e. curves whose normalization is 
hyperelliptic and such that points lying over the node are conjugate.
 
Curves with separating $A_7$ singularities (smooth 
$(1,1)$ and $(2,2)$ curves meeting with multiplicity $4$ at a single 
point) replace curves in $\Delta_1$ with a hyperelliptic 
genus $3$ component.

Curves with $A_6$ singularities 
replace curves in $\Delta_1$ with a hyperelliptic 
genus $3$ component attached at a Weierstrass point.

\subsubsection{} 
Finally, using Proposition \ref{P:stable-limits} we see that:
Curves with non-separating $A_5$-singularities replace hyperelliptic admissible
covers with two irreducible components of genus $1$ and $2$. Curves with $A_4$ singularities 
replace curves with Weierstrass genus $2$ tails, 
i.e., curves that are a nodal union of two genus $2$ curves in a point which is a Weierstrass point on one of the components. 
Curves with $A_3$ singularities replace curves with elliptic bridges.
Curves with $A_2$ singularities replace curves with elliptic tails.
\subsubsection{} 
We summarize this section by collecting the observations regarding which singular curves in $M$ replace which 
geometrically meaning loci in $\M_4$ under the rational map $f\co \M_4 \dra M$. The list of singular curves discussed in this section,
together with their varieties of stable 
limits is given in Table \ref{table-replacement}. We note 
that this list is not exhaustive and, for example, does not include all possible boundary strata flipped by $f$.
{\small
\begin{table}[htb]
\centering
\renewcommand{\arraystretch}{1.3}
\begin{tabular}{| c | c |}
\hline
Singularity type introduced   & Locus removed    \\
\hline
\hline
$A_2$                & {elliptic tails attached nodally}\\
\hline
 $A_3$                & {elliptic bridges attached nodally } \\
\hline
$A_4$                & {genus $2$ tails attached nodally at a Weierstrass point}  \\
\hline
non-separating $A_5$                & {genus $2$ bridges attached nodally at conjugate points}         \\
\hline
$\text{the maximally degenerate $A_5$-curve}$  & $\Delta_2$\\
\hline
$A_6$                & {hyperelliptic genus $3$ tails attached nodally at a Weierstrass point}         \\
\hline
$A_7$                & {hyperelliptic genus $3$ tails attached nodally}         \\
\hline
$D_4$    &    {elliptic triboroughs}   \\
\hline
{double conics}    & {curves in $\Delta_0$ with a hyperelliptic normalization}\\
\hline
$A_{8}, A_{9}$ & {hyperelliptic curves} \\
\hline
{the triple conic} & {the Petri divisor $P$} \\
\hline
\end{tabular}
\medskip
\caption{Replacing varieties of stable limits by singular curves}
\label{table-replacement}
\end{table}
}

\section{Birational geometry of $\M_4$}
 As before, $V=\vert\O_{\PP^1\times\PP^1}(3,3)\vert$,  
 $G=\SL(2)\times \SL(2)\rtimes \ZZ_2$, and $M=V^\ss\gitq G$. The GIT quotient 
 $\phi\co V^\ss\ra M$
 has been described in detail in the previous section, where we have seen 
 that the natural map $f\co \M_4\dra M$ is birational.
 We now describe how $M$ fits into the Hassett-Keel program for $\M_4$.
\begin{prop}\label{P:contraction}
The rational map $f\co \M_{4} \dra M$ is a contraction, i.e.
$f^{-1}$ does not contract divisors.
\end{prop}
\begin{proof}
It suffices to show that $f^{-1}$ is defined away from codimension $2$ locus. Consider 
the $G$-invariant closed subscheme  
$$
B:=\{[C]\in V \mid \text{$C$ has worse than nodal singularities}\}.
$$
Clearly, curves in $V^{\ss} \setminus B$ are moduli semistable. The existence of the stabilization morphism
implies that the morphism $V^{\ss} \setminus B \ra \M_4$ is well-defined. Since this morphism 
is $G$-invariant and $\phi\co V^{\ss}\ra M$ 
is a GIT quotient,
we see that $f^{-1}$ is a regular morphism on 
$M\setminus \phi(B)$.
As the discriminant divisor in $V$ is
irreducible with the geometric generic point a nodal curve, 
we deduce that every irreducible component of 
$B$ has codimension at least $2$ inside $V^\ss$. Thus 
$\text{codim}(\phi(B), M)\geq 2$.
\end{proof}

\begin{prop}\label{P:delta-lambda-V}
On $V\simeq \PP^{15}$, we have 
$\delta=\O_V(34)$ and $\lambda=\O_V(4)$.
\end{prop}
\begin{proof}
Let $\C \hookrightarrow \PP^1\times \PP^1\times V$, together with the 
projection $\mathrm{pr}_3\co \C\ra V$, be the universal $(3,3)$ curve. 
Set $H_i=\mathrm{pr}_{i}^*\O(1)$.
Since $\C$ is a smooth $(3,3,1)$-divisor on $\PP^1\times \PP^1\times V$, 
we obtain by adjunction 
that $\omega_{\C/V}=\O_{\C}(1,1,1)$. 
By pushing forward via $\mathrm{pr}_{3}$ the exact sequence 
\[
0\ra \O(-2,-2,0) \ra \O(1,1,1)\ra \O_{\C}(1,1,1) \ra 0,
\]
we deduce that 
\[
\lambda=c_1((\mathrm{pr}_{3})_*\O_{\C}(1,1,1))=c_1((\mathrm{pr}_{3})_*\O(1,1,1))=4H_3.
\]
We also compute that 
\[
\kappa=(\mathrm{pr}_{3})_*(\omega_{\C/V}^2)
=(\mathrm{pr}_{3})_*\left((3H_1+3H_2+H_3)(H_1+H_2+H_3)^2\right) 
=14H_3.\]
Using Mumford's formula $\lambda=(\kappa+\delta)/12$,
we conclude that $\delta=34H_3$. The claim follows.
\end{proof}
\begin{corollary}\label{C:lambda-delta}
On $M=V^\ss\gitq G$, 
we have $\delta=\O_M(34)$ and $\lambda=\O_M(4)$, 
where $\O_M(1)$ is the GIT polarization coming from $\O_V(1)$.
\end{corollary}
\begin{proof}
Since $\lambda$ and $\delta$ are $G$-invariant divisor classes, 
the results of Proposition \ref{P:delta-lambda-V}
descend to $M$.
\end{proof}

\subsection{Proof of Main Theorem}
\label{S:proof-main-theorem}

We have seen that $f\co \M_4\dra M$ is a birational contraction in Proposition \ref{P:contraction}. Theorem \ref{T:determinacy}
below shows that $f$ contracts the Petri divisor $P$, and the boundary divisors $\Delta_1$, $\Delta_2$. Thanks to the GIT analysis in Section \ref{S:GIT},
$f$ is well-defined at all points of $\M_4\setminus (\Delta \cup P)$ and is well-defined at the generic point of $\Delta_0$. 
It follows that $f$ is defined at the generic point of every irreducible divisor in $\M_4$ with the exception of $P$, $\Delta_1$, and $\Delta_2$.

\begin{lemma}[Petri divisor]\label{L:Petri-divisor}
 The divisor class of the Petri divisor 
is 
$$
P=17\lambda-2\delta_0-7\delta_1-9\delta_2.
$$
\end{lemma}
\begin{proof}
This is an instance of Theorem 2 of \cite{EHPetri} for $g=4$. 
\end{proof}

\begin{prop}\label{L:discrepancy}
We have 
\begin{align}
f^*\lambda=f^*f_*\lambda &=\lambda+\delta_1+3\delta_2+7P, \\
f^*\delta=f^*f_*\delta &=\delta_0+12\delta_1+30\delta_2+60P=\delta+11\delta_1+29\delta_2+60P.
\end{align}
\end{prop}

\begin{proof}
We begin by writing
\begin{align*}
f^*\lambda &=\lambda+a_0\delta_1+b_0\delta_2+c_0P, \\
f^*\delta &=\delta_0+a_1\delta_1+b_1\delta_2+c_1P.
\end{align*}

We will now find indeterminate coefficients by using judiciously chosen test families. 
A care needs to be exercised to use curves $T\subset \M_4$ such that $f$ is defined along $T$
and such that $f(T)$ is a point. We prove that our test families satisfy these requirements in 
Theorem \ref{T:determinacy} below.

\subsection*{Test families}
\subsubsection{Elliptic tails\em:} Our first test family $T_1$ is obtained by attaching the family 
of varying elliptic tails to the general pointed curve of genus $3$. 
The rational map $f\co \M_4\dra M$ is defined in the neighborhood of $T_1$ and contracts
$T_1$ to a point by Theorem \ref{T:determinacy} (2). We also have:
\begin{align}
\lambda\cdot T_1=1, \quad \delta_0\cdot T_1=12, \quad \delta_1\cdot T_1=-1, \quad \delta_2\cdot T_1=P\cdot T_1=0.
\end{align}
It follows that $a_0=1$ and $a_1=12$. 

\subsubsection{Genus $2$ tails\em:} Consider now
the family $T_2$ of irreducible genus $2$ tails attached 
at non-Weierstrass points.
The intersection numbers of this family are standard and are written down in \cite{handbook} Table 4.2.2:
\begin{align}
\lambda\cdot T_2=3, \quad \delta_0\cdot T_2=30, \quad \delta_1\cdot T_2=0, \quad \delta_2\cdot T_2=-1, 
\quad P\cdot T_2=0.
\end{align}
By Theorem \ref{T:determinacy} (3) we have that $f$ is defined in the neighborhood of $T_2$ 
and $f(T_2)$ is a point. It follows that $b_0=3$ and $b_1=10$.

\subsubsection{Petri curves\em:} We now take the family $T_3$ 
of Petri-special curves on $\PP(1,1,2)$ defined by the 
weighted homogeneous equation 
\[
y^3=x^6+axyz^3+bz^6,
\]
where $[a:b]\in \PP(1,2)$. Evidently, these curves are at worst nodal and avoid the vertex $[0:0:1]$
of the singular quadric. Thus $f$ is defined in the neighborhood of $T_3$
by Theorem \ref{T:determinacy} (1) and $f(T_3)$ is a point.
The intersection numbers of this family are computed in \cite[Proposition 6.6]{afs} and are as follows:
\begin{align}
\lambda\cdot T_3=7, \quad \delta_0\cdot T_3=60, \quad \delta_1\cdot T_3=\delta_2\cdot T_3=0.
\end{align}
From Lemma \ref{L:Petri-divisor}, compute $P\cdot T_3=119-120=-1$. Thus $c_0=7$ and $c_1=60$.
\end{proof}

\begin{corollary}\label{logMMP}
We have $\M_4(\alpha)\simeq M$ for all $\frac{8}{17}<\alpha\leq 29/60$. Moreover, $\M_4(8/17)$ is a point and 
$\M_4(\alpha)=\varnothing$ if $\alpha<\frac{8}{17}$.
\end{corollary}
\begin{proof}
Let $f\co \M_4 \dra M$ be the rational contraction. Then 
$$f_*(K_{\Mg{4}}+\alpha\delta)=f_*(13\lambda-(2-\alpha)\delta)=13\lambda-(2-\alpha)\delta.$$
We now compute using Proposition \ref{L:discrepancy}
\begin{align*}
(K_{\Mg{4}}+\alpha\delta)-f^*f_*(K_{\Mg{4}}+\alpha\delta)&=(13\lambda-(2-\alpha)\delta)-f^*f_*(13\lambda-(2-\alpha)\delta) \\ 
&=-13(\delta_1+3\delta_2+7P)+(2-\alpha)(11\delta_1+29\delta_2+60P) \\
&=(29-60\alpha)P+(19-29\alpha)\delta_2+(9-11\alpha)\delta_1.
\end{align*}
This is an effective exceptional divisor as long as $\alpha\leq 29/60$. It follows
that for $\alpha\leq 29/60$:
\begin{align*}
\M_4(\alpha) &=\proj \bigoplus_{m \ge 0} \HH^0(\Mg{g}, m( K_{\Mg{g}} + \alpha \delta))
\\
&= \proj \bigoplus_{m \ge 0} \HH^0(M, mf_*(K_{\Mg{g}} + \alpha \delta)) \\
&=\proj \bigoplus_{m \ge 0} \HH^0(M, m(13\lambda-(2-\alpha)\delta))
=\proj \bigoplus_{m \ge 0} \HH^0(M, \O_{M}(34\alpha-16)),
\end{align*}
where we have used Corollary \ref{C:lambda-delta} in the last step.
The statement now follows from the fact that $\O_{M}(34\alpha-16)$
is ample on $M$ for $\alpha>8/17$ and is a zero line bundle for $\alpha=8/17$. 
\end{proof}
\begin{corollary}[cf. Theorem \ref{T:moving-slope}]\label{C:moving-cone}
There is a moving divisor of class $60\lambda-7\delta_0-24\delta_1-30\delta_2$ on $\M_4$. 
Furthermore, any moving divisor 
$D=a\lambda-b_0\delta_0-b_1\delta_1-b_2\delta_2$ satisfies
$a/b_0\geq 60/7$. In particular, the moving slope of $\M_4$ is $60/7$. 
\end{corollary}
\begin{proof}
By Proposition \ref{L:discrepancy}, 
$$f^*(60\lambda-7\delta)=60\lambda-7\delta_0-24\delta_1-30\delta_2.$$
By Corollary \ref{C:lambda-delta} the divisor $60\lambda-7\delta=\O_M(2)$ is ample. Since $f$ is a rational contraction,
the divisor $f^*(60\lambda-7\delta)$ is
moving on $\M_4$.

Suppose now $D$ is a moving divisor. Since families $T_1, T_2, T_3$ constructed in the
proof of Proposition \ref{L:discrepancy} are covering families for divisors $\Delta_1, \Delta_2$, 
and $P$, respectively, we have
\begin{align*}
D\cdot T_3\geq 0 &\Longrightarrow \frac{a}{b_0}\geq \frac{60}{7}, \\
D\cdot T_2\geq 0 &\Longrightarrow 3a-30b_0+b_2\geq 0, \\
D\cdot T_1\geq 0 &\Longrightarrow a-12b_0+b_1\geq 0. \\
\end{align*}
The statement follows.
\end{proof}

\subsection{A theorem on indeterminacy locus}
Here, we describe loci where the rational map 
$f\co \M_4\dra M=V^\ss\gitq G$ is regular,
completing the proof of Proposition \ref{L:discrepancy}. 
Recall that $\phi\co V^\ss \ra M$ is the GIT quotient.
\begin{definition}\label{D:stable-limits}
Let $\M_4 \stackrel{p}{\longleftarrow} Z 
\stackrel{q}{\longrightarrow} M$ be the graph of $f$. Recall that
the variety of stable limits of $[X]\in V^\ss$ is $\Tl_{X}=p(q^{-1}(\phi([X])))$. 
We now define the {\em variety of GIT-semistable
limits} of $[C]\in \M_4$ to be $\D_{C}:=q(p^{-1}([C]))$.
\end{definition}
\begin{remark}
We allow $[X]$ to be a non-closed point of the GIT stack
$[V^\ss / G]$.
\end{remark}

\begin{lemma}\label{L:determinacy-f}
Suppose that for $[C]\in \M_4$, we have $[X]\in \D_C$. Suppose
further that 
for every closed point $[X']$ in a small punctured neighborhood 
of $[X]$ in 
$M$, we have that $[C]\notin \Tl_{X'}$. Then $f$ is defined
at $[C]$ and $f([C])=[X]$. 
\end{lemma}
\begin{proof}
Because both $\M_4$ and $M$ are normal varieties (the latter
by \cite[Theorem 1.1]{GIT}), $\D_C$ is connected.
Using the obvious implication
 $$[X']\in \D_C \Longrightarrow [C]\in \Tl_{X'},$$
we conclude that $\D_C$ does not meet a small punctured neighborhood 
of $[X]$.
It follows that $\D_C=[X]$. Thus, $f$ is defined at $[C]$ and $f([C])=[X]$. 
\end{proof}

\begin{definition} We define by $P^\circ \subset \M_4$ 
the locally closed subset of stable curves whose canonical embedding lies on a singular 
quadric, but which do not pass through the vertex of the quadric. 
\par 
We define 
by $\Delta_2^\circ \subset \M_4$ the locally closed subset of stable curves $[E_1\cup E_2]\in \Delta_2$
such that $E_1$ and $E_2$ are irreducible and $E_1\cap E_2$ is not
a Weierstrass point\footnote{A Weierstrass point of an irreducible 
stable curve of genus $2$ is a ramification point
of the canonical $2:1$ map onto $\PP^1$.} either on $E_1$ or
$E_2$. 
\end{definition}

\begin{remark} If $[C]\in P^\circ \cap \Delta_0$ 
is a curve with a single node,
then $\widetilde{C}$ is not hyperelliptic. Conversely,
if $[C]\in \Delta_0\setminus \overline{H}_4$ 
is a curve with a single node and 
$\widetilde{C}$ is hyperelliptic, then $[C]\in P\setminus P^\circ=\Delta_0^{hyp}$ because
the node maps under the canonical embedding to the vertex of the singular quadric containing $C$. 
\end{remark}

\begin{theorem}\label{T:determinacy}
 The rational map $f\co \M_4\dra M$ is regular at the following points:
\begin{enumerate}
\item
All curves in $P^\circ$.  Moreover, $f$ maps $P^\circ$ to the triple conic.
\item All curves in $\Delta_1$
whose genus $3$ component is the general curve in 
$\Mg{3,1}$. 
Moreover, if 
the pointed genus $3$ component is fixed, all curves with varying elliptic tails are mapped to the same 
cuspidal curve in $M$.
\item All curves in $\Delta_2^\circ$. 
Moreover, $f$ maps $\Delta_2^\circ$ to the maximally degenerate $A_5$-curve.
\end{enumerate} 
\end{theorem}
\begin{proof} To show that $f$ is defined at a point of $\M_{4}$,
we employ three different techniques: In the case of $P^\circ$,
we define the map explicitly; in the case of $\Delta_1$, we use 
the moduli space of pseudostable curves (see \cite{Schubert, hassett-hyeon_contraction}); 
lastly, in the case of $\Delta_2^\circ$, we use
varieties of stable limits.

\subsubsection*{Curves in $P^\circ$\em:}
Suppose $C$ is a stable curve lying on a rank three quadric 
and avoiding the vertex. We prove that $f$ is defined at $C$ by
showing that for {\em every} smoothing of $C$ away from the Petri locus,
the GIT-semistable
limit is the triple conic. Indeed, let $\C=\{C_t\}$ be a smoothing of $C$, 
with $C_0=C$ and $C_t$ smooth Petri-general curves for all $t\neq 0$.
Realize
$\{C_t\}$ as a family of canonically embedded curves by choosing 
a trivialization of the Hodge bundle. Associated to this family of 
canonical curves is the family of quadrics $\{Q_t\}$, with $Q_0$ a singular
quadric and $Q_t$ a smooth quadric for $t\neq 0$. We choose 
coordinates so that the equation for $Q_t$ is
$$
Q_t: \{z_0^2+z_1^2+z_2^2+t^{a}z_{3}^2=0\},
$$
where we arrange $a$ to be even using a finite base
change. Consider now the one-parameter family $\rho\co \CC[t, t^{-1}]\ra \mathrm{GL}(4)$ given by $\rho(t)=\text{diag}(1,1,1,t^{a/2-1})$. Then the 
family of canonical curves 
$C'_t:=\rho(t)(C_t)$ lies on the family of quadrics
defined by the equation 
$$
Q'_t:\{z_0^2+z_1^2+z_2^2+t^2z_3^2=0\}.
$$
Abstractly, $C'_t\simeq C_t$ for $t\neq 0$ and so the stable limit
of $\{C'_t\}$ at $t=0$ is still $C_0$. The flat limit 
$C'_0:=\lim_{t\to 0} C'_t$ remains $C_0$ if $a=2$, and is the triple conic $z_3=0$
on $Q'_0=Q_0$ if $a\geq 4$. In either case, $C_0'$
does not pass through the vertex of $Q'_0$.
Consider now the blow-up $\X':=\Bl_{p} \X$ of the total space $\X$ of 
$\{Q_t\}$ at $p=[0:0:0:1]$. The exceptional divisor $E$ is isomorphic to $\PP^1\times\PP^1$ and
meets the strict transform of $Q'_0$ in a smooth 
conic $O$. The strict transform of $Q'_0$ 
is isomorphic to $\FF_2$, with 
$O$ being a $(-2)$ curve. 
We now blow-down $\FF_2$ down to $O$.
The resulting threefold is the total space of the family 
of $\PP^1\times\PP^1$'s
with central fiber $E$.
The flat limit of $\{C'_t\}_{t\neq 0}$ in $E$ 
is now the triple conic $O$.
\begin{remark} It is clear what goes wrong when $[C_0]\in \Delta_0^{hyp}$, i.e. when $C_0$ 
passes through the vertex of $Q_0$. The exceptional divisor $E$
then meets the strict transform $\overline{\{C_t\}}$ in a conic $O'$.
As a result, after the blow-down, the flat limit of 
$\{C'_t\}_{t\neq 0}$ in $E$
is a union of the double conic $O$ and $O'$. 
The indeterminacy of $f$ along $\Delta_0^{hyp}$ 
arises because the cross-ratio (see Remark \ref{R:cross-ratio})
of the double conic curve depends on the smoothing $\C$. 
\end{remark}

\subsubsection*{Curves with elliptic tail\em:} 
Since there exists 
a morphism $\M_4 \ra \M_4^{\, ps}$ (see \cite{hassett-hyeon_contraction}), 
it suffices to show
that the morphism from $\M_4^{\, ps}$ to $M$ is well-defined at the general cuspidal curve.
This immediately follows from the fact that a cuspidal curve $C$ whose
pointed normalization is the general curve in $\Mg{3,1}$ is 
embedded by $\omega_C$ into a smooth quadric in $\PP^3$.

\subsubsection*{Curves in $\Delta_2^\circ$\em:}
Consider the maximally degenerate $A_5$-curve $X$ on
$\PP^1\times\PP^1$. By Proposition \ref{P:stable-limits} 
and Proposition \ref{P:versality}, the variety of stable limits
of $X$ contains all curves $[E_1\cup E_2]\in \Delta_2$
 such that $E_1\cap E_2$ is not a Weierstrass point on 
 either $E_1$ or $E_2$. 
 
 Consider now a small deformation $X'$ of $X$. 
 If one of the $A_5$-singularities
 of $X$ is preserved, then the singularity remains separating on 
 $X'$. It follows that $X'$ is a union of a $(1,0)$-ruling and a residual 
 $(2,3)$-curve tangent to the ruling with multiplicity $3$ at a smooth 
 point. Such a curve is necessarily 
 defined by Equation \eqref{E:semistable31} and hence $[X']=[X]\in M$.
 Suppose both $A_5$ singularities are smoothed in $X'$. Then 
 $X'$ has at worst $A_1, A_2, A_3, A_4$ singularities. 
 By Proposition \ref{P:stable-limits}, the tails of stable limits arising from $A_2$ and $A_3$ 
 singularities can have irreducible components of arithmetic
 genus at most $1$ and the tails of stable limits arising from $A_4$
 singularities can have irreducible components of 
 arithmetic genus $2$ only if the component is attached to the rest of the curve at a Weierstrass point. This shows that $\Tl_{X'}\cap \Delta_2^\circ= \varnothing$. We are done by Lemma \ref{L:determinacy-f}.
\end{proof}

We finish the discussion of the indeterminacy locus of the rational map $f$ by proving the following lemma used in the proof of Theorem \ref{T:flip}.
\begin{lemma}\label{L:stable-limits-double-conic}
Let $C\subset \PP^1\times \PP^1$ be a double conic. 
Then the variety of stable limits of $C$ is contained in $\Delta_0$.
\end{lemma}
\begin{proof}
Consider a smooth quadric 
$\PP^1\times\PP^1 \simeq Q \subset \PP^3$ and choose projective coordinates $[z_0:z_1:z_2:z_3]$
on $\PP^3$ so that the double conic $C$ is cut out by $z_0^2z_1=0$ on $Q$ and so that $[1:0:0:0]\notin Q$. 
Consider now a one-parameter subgroup $\rho\co \spec \CC[t,t^{-1}] \ra \PGL(4)$ acting by 
\[
t\cdot [z_0:z_1:z_2:z_3]=[z_0: tz_1: tz_2: tz_3].
\]
Then $C_0:=\lim_{t\to 0}\rho(t)\cdot C$ is a genus $4$ curve lying on a singular quadric $Q_0=\lim_{t\to 0}\rho(t)\cdot Q$
with a vertex at $[1:0:0:0]$.
Evidently, $C_0$ is a union of a double conic (note that the double conic of $C$ is fixed under $\rho(t)$) and two rulings of $Q_0$ meeting 
in a node at $[1:0:0:0]$. 
Since $C_0$ is a flat degeneration of $C$ in an isotrivial family, it follows that the variety of stable limits of $C$
is contained inside the variety of stable limits of $C_0$. It remains to observe that since $C_0$ has a node and 
the partial normalization of $C_0$ at this node has arithmetic genus $3$, every stable limit of $C_0$ also has a non-separating node.
This finishes the proof.
\end{proof}

\begin{theorem}[Flip of the hyperelliptic locus]\label{T:flip}
The hyperelliptic locus $\overline{H}_4$ is flipped by $f$ to the one dimensional locus
$A:=\overline{\{\text{curves with an $A_8$ singularity}\}}$, i.e.
the total transform of the generic point of $\overline{H}_4$ is $A$.
\end{theorem}
\begin{proof}
Recall from Section \ref{S:A-curves} that $A\subset M$ 
is a curve passing 
through the triple conic and through the unique $A_9$-point, and
smooth away from these two points. (It is not 
hard to see that $A$ is a rational curve, but we do not use this
fact.) As we have already observed, the variety of stable
limits of every $A_8$ curve is $\overline{H}_4$ by 
Proposition \ref{P:stable-limits} $(\mathrm{A_{\text{even}}})$. 
It remains to show that for any curve $X\in M\setminus A$, the 
variety of stable limits $\Tl_{X}$ does not pass through the 
generic point of $\overline{H}_4$. This is analogous to the proof
of Theorem \ref{T:determinacy} (3): Every closed semistable curve not
in $A$ is either a double conic, or a $D_4$-curve, or has at worst $A_7$ singularities. But by
Proposition \ref{P:stable-limits}, the general hyperelliptic 
curve does not lie in the variety of stable limits of a $D_4$-curve or a $(3,3)$ curve
with at worst $A_7$ singularities. Finally, the variety of stable limits of a double conic is contained in $\Delta_0$ 
by Lemma \ref{L:stable-limits-double-conic} and 
hence also does not contain the general hyperelliptic curve.
\end{proof}
\begin{remark}
As we have seen in Sections \ref{S:A67-curves}, 
the flipping loci of special
closed subvarieties of $\overline{H}_4$ 
do lie outside of $A$. For example, hyperelliptic
curves in $\Delta_1$ are flipped to curves with $A_6$ singularities, etc.
\end{remark}

\section{Concluding remarks}

We conclude that the log canonical model $\M_4(\alpha)$ satisfies
the modularity principle for the log MMP for $\M_4$ for $\alpha\in [0,29/60)$. 
 The final non-trivial log canonical model $\M_4(29/60)$ 
also exhibits behavior that we expect of log canonical models with 
$\alpha>29/60$. Namely, for 
all $8/17< \alpha<5/9$:
\begin{enumerate}
\item Hyperelliptic curves are replaced by curves with $A_8$ (and $A_9$) singularities.
\item Curves with elliptic triboroughs 
are replaced by curves with two $D_4$ singularities. 
\item General curves in $\Delta_2$ are replaced by a maximally degenerate $A_5$-curve.
\end{enumerate}
We finish by noting that although this paper confirms the above assertions only for $8/17<\alpha\leq 29/60$, some of its results
are readily extended to higher values of $\alpha$. For example,
the canonically embedded
maximally degenerate $A_5$-curve (see \eqref{E:semistable31})
is defined by the ideal 
$(z_0z_3-z_1z_2, z_1^2z_3+z_2^2z_0)$ in $\PP^3$ 
and has a semistable $m^{th}$ Hilbert point for 
all $m\geq 3$ \cite{AFS-even}. 
This suggests that the maximally
degenerate $A_5$-curve replaces $\Delta_2$ in all of GIT quotients
$\Hilb_{4,1}^{\, m, \ss}\gitq \SL(4)$, and hence in
all log canonical models $\M_4(\alpha)$ with $8/17<\alpha\leq 5/9$.

\subsection*{Acknowledgements}
This work was motivated by a question of Gabi Farkas, 
to whom we are grateful. We also thank an anonymous referee for a careful reading of the paper and 
useful suggestions on how to improve the exposition.
\bibliographystyle{alpha}
\bibliography{stability-bib}

\def\cprime{$'$} \def\cprime{$'$} \def\cprime{$'$}
\begin{thebibliography}{ASvdW10}

\bibitem[AFS10]{afs}
Jarod Alper, Maksym Fedorchuk, and David Smyth.
\newblock Singularities with {$\GG_m$}-action and the log minimal model program
  for {$\overline{M}_g$}, 2010.
\newblock {\tt arXiv:1010.3751v2 [math.AG]}.

\bibitem[AFS11]{AFS-even}
Jarod Alper, Maksym Fedorchuk, and David Smyth.
\newblock Finite {H}ilbert stability of canonical curves, {II}. {T}he
  even-genus case, 2011.
\newblock {\tt arXiv:1110.5960 [math.AG]}.

\bibitem[Arn76]{arnold-inventiones}
V.~I. Arnol{\cprime}d.
\newblock Local normal forms of functions.
\newblock {\em Invent. Math.}, 35:87--109, 1976.

\bibitem[ASvdW10]{asw}
Jarod Alper, David Smyth, and Fred van~der Wyck.
\newblock Weakly proper moduli stacks of curves.
\newblock {\tt arXiv:1012.0538 [math.AG]}, 2010.

\bibitem[CMJL11a]{laza-et-al}
Sebastian Casalaina-Martin, David Jensen, and Radu Laza.
\newblock The geometry of the ball quotient model of the moduli space of genus
  four curves, 2011.
\newblock {\em Proceedings of the UGA Conference on Moduli and Vector Bundles},
  to appear. Available at {\tt arXiv:1109.5669v1 [math.AG]}.

\bibitem[CMJL11b]{laza-et-al-2}
Sebastian Casalaina-Martin, David Jensen, and Radu Laza.
\newblock Variation of {GIT} for genus 4 canonical curves, 2011.
\newblock {P}reprint.

\bibitem[EH83]{EHPetri}
D.~Eisenbud and J.~Harris.
\newblock A simpler proof of the {G}ieseker-{P}etri theorem on special
  divisors.
\newblock {\em Invent. Math.}, 74(2):269--280, 1983.

\bibitem[Far10a]{farkas-gieseker}
Gavril Farkas.
\newblock Rational maps between moduli spaces of curves and {G}ieseker-{P}etri
  divisors.
\newblock {\em J. Algebraic Geom.}, 19(2):243--284, 2010.

\bibitem[Far10b]{farkas-personal}
Gavril Farkas.
\newblock Private communication, July 2010.

\bibitem[Fed10]{fedorchuk-AD}
Maksym Fedorchuk.
\newblock Moduli spaces of hyperelliptic curves with {A} and {D} singularities,
  2010.
\newblock {\tt arXiv:1007.4828v2 [math.AG]}.

\bibitem[FJ11]{FedJen}
Maksym Fedorchuk and David Jensen.
\newblock Stability of {$2^{nd}$} {H}ilbert points of canonical curves, 2011.
\newblock {P}reprint, http://math.columbia.edu/$\sim$mfedorch/2nd-hilbert.pdf.

\bibitem[FS10]{handbook}
Maksym Fedorchuk and David Smyth.
\newblock Alternate compactifications of moduli spaces of curves, 2010.
\newblock {\em Handbook of Moduli}, edited by G. Farkas and I. Morrison, to
  appear. {\tt arXiv:1012.0329v2 [math.AG]}.

\bibitem[Has00]{Hassett-stable}
Brendan Hassett.
\newblock Local stable reduction of plane curve singularities.
\newblock {\em J. Reine Angew. Math.}, 520:169--194, 2000.

\bibitem[HH08]{hassett-hyeon_flip}
Brendan Hassett and Donghoon Hyeon.
\newblock Log minimal model program for the moduli space of curves: the first
  flip, 2008.
\newblock {\tt arXiv:0806.3444 [math.AG]}.

\bibitem[HH09]{hassett-hyeon_contraction}
Brendan Hassett and Donghoon Hyeon.
\newblock Log canonical models for the moduli space of curves: the first
  divisorial contraction.
\newblock {\em Trans. Amer. Math. Soc.}, 361(8):4471--4489, 2009.

\bibitem[HL10a]{hyeon-lee_genus4}
D.~{Hyeon} and Y.~{Lee}.
\newblock {Birational contraction of genus two tails in the moduli space of
  genus four curves I}, 2010.
\newblock {\tt arXiv:1003.3973 [math.AG]}.

\bibitem[HL10b]{hyeon-lee_genus3}
Donghoon Hyeon and Yongnam Lee.
\newblock Log minimal model program for the moduli space of stable curves of
  genus three.
\newblock {\em Math.Res.Lett.}, 17(04):625--636, 2010.

\bibitem[Jaw87]{J10}
P.~Jaworski.
\newblock {Decompositions of parabolic singularities of one level.}
\newblock {\em Mosc. Univ. Math. Bull.}, 42(2):30--35, 1987.

\bibitem[{Jen}10]{jensen-genus4}
D.~{Jensen}.
\newblock {Birational Contractions of {$\M_{3,1}$} and {$\M_{4,1}$}}, 2010.
\newblock {\em Trans. Amer. Math. Soc.}, (to appear). Available at {\tt
  arXiv:1010.3377v2 [math.AG]}.

\bibitem[MFK94]{GIT}
D.~Mumford, J.~Fogarty, and F.~Kirwan.
\newblock {\em Geometric invariant theory}, volume~34 of {\em Ergebnisse der
  Mathematik und ihrer Grenzgebiete (2) [Results in Mathematics and Related
  Areas (2)]}.
\newblock Springer-Verlag, Berlin, third edition, 1994.

\bibitem[Sch91]{Schubert}
David Schubert.
\newblock A new compactification of the moduli space of curves.
\newblock {\em Compositio Math.}, 78(3):297--313, 1991.

\end{thebibliography}

\end{document}